\documentclass{article}
\usepackage{graphicx} 
\usepackage{amsmath}
\usepackage{amssymb}
\usepackage{amsthm}
\usepackage{amssymb}
\usepackage{mathdots}
\usepackage{enumitem}
\usepackage{tabularx}
\usepackage{color}
\title{Labeled histories for multifurcating trees}
\usepackage[margin=1in]{geometry}
\usepackage{stmaryrd}
\usepackage{lscape}

\renewcommand{\geq}{\geqslant}
\renewcommand{\leq}{\leqslant}

\newtheorem{theorem}{Theorem}
\newtheorem{prop}[theorem]{Proposition}
\newtheorem{corollary}[theorem]{\bf Corollary}
\newtheorem{lemma}[theorem]{Lemma}
\newtheorem{defi}[theorem]{Definition}
\newtheorem{conjecture}[theorem]{Conjecture}

\title{Maximally probable tree topologies with $r$-furcation}
\author{Emily H.~Dickey$^*$ \& Noah A.~Rosenberg\thanks{Department of Biology, Stanford University, Stanford, CA 94305 USA}}
\date{\today}
\begin{document}
\maketitle

\begin{abstract} \noindent For a specific rooted labeled tree topology, a labeled history is a sequence of branchings that give rise to that labeled topology as it unfolds over time. Here, for $r$-furcating trees, we use a connection with Huffman trees from information theory to identify maximally probable rooted trees---unlabeled $r$-furcating topologies whose labelings each have a number of labeled histories greater than or equal to those of all other labeled topologies. Our characterization of the unique maximally probable $r$-furcating unlabeled topology generalizes the Harding--Hammersley--Grimmett result identifying the maximally probable bifurcating unlabeled topology, and it provides a new proof for that result. We present a conjecture for the maximally probable $r$-furcating unlabeled topology if labeled histories are tabulated allowing for simultaneous branching events across multiple internal nodes of a tree.
\end{abstract}

\vspace{.2cm}

\vspace{.2cm}
\noindent{\bf Author for correspondence:} Noah A.~Rosenberg. Email: noahr@stanford.edu.

\section{Introduction}

The concept of a \textit{labeled history} arises when considering phylogenetic models that produce trees as a consequence of evolutionary processes that unfold over time~\cite{King23, Steel16}. For a leaf-labeled tree with $n$ leaves that have mutually distinct labels, a labeled history represents the sequence of branching events that have given rise to the labeled tree structure. Labeled histories are useful particularly in the setting of the Yule--Harding model on bifurcating labeled topologies~\cite{King23, Steel16}, under which each of the $n! \, (n-1)!/2^{n-1}$ possible labeled histories~\cite{Edwards1970} is equally likely to be produced by the evolutionary process. 

For a fixed number of leaves $n \geq 2$, a \textit{maximally probable} tree is an unlabeled topology whose labelings each possess a number of labeled histories greater than or equal to those of any other unlabeled topology~\cite{Degnan06}. Proving a conjecture of Harding \cite{Harding1971}, Hammersley \& Grimmett \cite{Hammersley74} characterized the unique maximally probable tree among rooted bifurcating trees with a fixed number of leaves $n$.

Generalizing from bifurcating trees, we have recently been developing the combinatorics of labeled histories in the setting of $r$-furcating trees~\cite{Dickey25, Dickey26}---which arise in contexts such as adaptive radiation, large family sizes, and pathogen transmission. In particular, in \cite{Dickey25}, we obtained a variety of results concerning labeled histories for $r$-furcating trees, also termed \emph{strictly} $r$-furcating trees, in which the number of immediate descendants of each internal node of a rooted tree is equal to $r$ for a fixed constant $r \geq 2$ shared by all internal nodes. Our results included the enumeration of the labeled histories for a fixed $r$-furcating labeled topology. We also presented a conjecture for the identity of the maximally probable tree shape with $r$-furcation.

Here, we identify the unique maximally probable tree in the setting of $r$-furcating trees, $r \geq 2$, proving our conjecture. That is, for $n \geq r$ and $n=w(r-1)+1$ for integers $w \geq 0$, we determine the shape of the unique unlabeled topology whose labelings possess more labeled histories than any other labeled topology. Our approach makes use of a connection to tree problems in information theory; specifically, we use the Huffman algorithm~\cite[Section 5.6, p.~92]{Cover05} and the trees it generates. Because our approach identifies the maximally probable tree for each $r \geq 2$, a consequence of our characterization is that we obtain a new proof of the Harding--Hammersley--Grimmett result describing the maximally probable bifurcating unlabeled topology. 

Section \ref{app:def} presents definitions concerning rooted trees. Section \ref{app:previous_results} reviews previous work on labeled histories and maximally probable trees, and it recasts the maximization problem as a minimization of a sum of logarithms. Section \ref{app:majorization} presents definitions and lemmas about majorization. Section \ref{app:huffman} discusses the Huffman algorithm from information theory and trees constructed through the algorithm. In particular, we show that the Huffman algorithm produces a unique tree topology on a uniform weight set. Section \ref{app:completing} combines results from Sections \ref{app:majorization} and \ref{app:huffman} to complete the characterization of the maximally probable $r$-furcating tree. Finally, Section \ref{app:withsim} presents conjectures for the maximally probable $r$-furcating tree in a setting that allows simultaneous branching. We conclude in Section \ref{app:discussion} with a discussion.

\section{Definitions}\label{app:def}

Definitions largely follow Dickey \& Rosenberg \cite{Dickey25}, tracing to Steel~\cite{Steel16} and King \& Rosenberg~\cite{King23}. For a rooted tree $T$, nodes are \emph{leaf nodes} or \emph{internal nodes}; the unique \emph{root node} is included among internal nodes. The number of internal nodes of $T$ is $w(T)$; the number of leaves is $|T|$. For a leaf-labeled tree $T$, the \emph{labeled topology} of $T$ is its topological structure together with its leaf labels, all of which are distinct from one another. The \emph{tree shape} or \emph{unlabeled topology} of $T$ is the topological structure without the leaf labels.

For two nodes $u,v$ of $T$, $u$ is \emph{descended} from $v$ and $v$ is \emph{ancestral} to $u$ if $v$ lies on the path from the root to $u$. A node is both ancestral to itself and descended from itself. 

For $r \geq 2$, in an \emph{$r$-furcating tree}, also termed \emph{strictly} $r$-furcating, each internal node has exactly $r$ immediate descendant nodes. A \emph{cherry node} has exactly two child nodes, both of which are leaves. The generalization for $r$-furcating trees is a \emph{broomstick node}, an internal node whose $r$ children are all leaves.

For a tree $T$ whose root has immediate subtrees $T_1, T_2, \ldots, T_r$, we write $T = T_1 \oplus T_2 \oplus \ldots \oplus T_r$. If some of the $T_i$ are identical, then we abbreviate the notation; for example, if $T_1 = T_2 = \ldots = T_j$, then we write $T = jT_1 \oplus T_{j+1} \oplus T_{j+2} \ldots \oplus T_r$. Empty subtrees are ignored, so that, for example, $T_1 \oplus T_2 \oplus \emptyset = T_1 \oplus T_2$. 
 
For a tree $T$, each node is associated with a \emph{time}. Leaves all have the same time. In the classic Yule--Harding model for bifurcating trees, internal nodes have mutually distinct times, and the tree has \emph{non-simultaneous} branching. Given a labeled topology for a rooted tree $T$ with $w$ internal nodes and non-simultaneous branching, a \emph{labeled history} for $T$ is a bijection $f$ from the set of internal nodes of $T$ to $\{1, 2, \ldots, w(T)\}$, so that if node $u$ is descended from node $v$ in $T$ and $u \neq v$, then $f(u) < f(v)$. A labeled history can be viewed as the temporal sequence of internal nodes, with the convention here that the numbers assigned to nodes increase backward in time along genealogical lines, and the root is assigned the value $w(T)$.

Among $r$-furcating labeled topologies with fixed number of leaves $n$ and fixed $r \geq 2$, a \emph{maximally probable} labeled topology is a labeled topology whose number of labeled histories is greater than or equal to that of all other labeled topologies~\cite{Degnan06}. Because each labeling of an unlabeled topology gives rise to the same number of labeled histories, we use \emph{unlabeled topologies} to indicate the maximally probable labeled topologies. In a slight abuse of terminology, we refer to unlabeled topologies as maximally probable, understanding that we refer to associated arbitrarily labeled topologies. The term \emph{maximally probable} arises from the Yule--Harding probability model, in which the probability of a specified (bifurcating) labeled topology is proportional to its number of labeled histories---so that labeled topologies with the most labeled histories are the most probable labeled topologies under the model. The Yule--Harding model is for bifurcating trees, but we continue to use the term \emph{maximally probable} in the setting of $r$-furcation: a maximally probable unlabeled topology is an unlabeled topology whose labelings possess a number of labeled histories greater than or equal to those of the labelings of all other unlabeled topologies. 

Although the definition does not imply that only one maximally probable unlabeled topology exists among $r$-furcating unlabeled topologies with $n$ leaves, we will see in Section \ref{app:completing} that with $r$-furcation at fixed $r \geq 2$, there exists a unique maximally probable unlabeled topology among $r$-furcating unlabeled topologies with $n$ leaves. We denote this unique maximally probable unlabeled topology by $U_n^*$, understanding that the associated $r$ will be clear from the context.

\section{Review of previous results}\label{app:previous_results}

For bifurcating trees, Hammersley \& Grimmett~\cite{Hammersley74} obtained the unique unlabeled tree topology whose labeled topologies possess the largest number of labeled histories among labeled topologies with $n$ leaves. Our main result, presented in Section \ref{app:completing}, is a generalization of this result to $r$-furcating trees, $r \geq 2$. The generalization provides a new proof of the result of Hammersley \& Grimmett~\cite{Hammersley74} and proves our Conjecture 13 of \cite{Dickey25}. First, we recall the maximally probable unlabeled tree topology in the bifurcating case.

\begin{theorem}[\cite{Hammersley74}]
\label{thm: hammersley}
The unique unlabeled topology whose labelings have the largest number of labeled histories among bifurcating labeled topologies with $n$ leaves takes the form $U_n^* = U_t^* \oplus U_{n-t}^*$, where for $n \geq 3$,
\begin{align*}
t = 2^{ \lfloor \log_2 (\frac{n-1}{3})\rfloor + 1}.
\end{align*}
\end{theorem}
$U_1^*$ and $U_2^*$ are specified trivially, each as the only unlabeled topology with the associated value of $n$. For $n \geq 3$, $U_n^*$ is formed by joining the maximally probable unlabeled topologies of two smaller sizes to a shared root. The values $t$ and $n-t$ in the theorem specify the numbers of leaves in the subtrees of the root of $U_n^*$. As $n$ proceeds from 3 to 16, the values of $(t,n-t)$ are equal to (1,2), (2,2), (2,3), (2,4), (4,3), (4,4), (4,5), (4,6), (4,7), (4,8), (8,5), (8,6), (8,7), and (8,8). 

To generalize Theorem~\ref{thm: hammersley} to $r$-furcation, consider $r \geq 2$, and let $n=w(r-1)+1$ for an integer $w \geq 1$, as the number of leaves in an $r$-furcating tree is 1 more than a multiple of $r-1$. The number of labeled histories $N(T)$ for an $r$-furcating labeled topology, $T$, with $n$ leaves and immediate subtrees $T_1, T_2, \ldots, T_r$ can be found recursively by~\cite[eq.~3.6]{Dickey25}
$$N(T) = {\frac{|T| - 1}{r-1}-1 \choose \frac{|T_1| - 1}{r-1}, \frac{|T_2|-1}{r-1}, \ldots, \frac{|T_r| - 1}{r-1}} N(T_1) N(T_2) \ldots N(T_r).$$
$N(T)$ can also be expressed in a closed form. Write $V^{0}(T)$ for the set of internal nodes of $T$, and let $m(v)$ be the number of leaves descended from internal node $v$.

\begin{prop} [\cite{Dickey25}, Prop.~8]
\label{prop:explicit_num_LH}
The number of labeled histories for an $r$-furcating labeled topology $T$ with $n$ leaves, $N(T)$, satisfies $N(T)=1$ for $n=1$, and for $n = w(r-1)+1$ with $w \geq 1$, 
\begin{align*}
N(T) = \frac{\big(\frac{n-1}{r-1}\big)!}{\prod_{v \in V^0(T)}\big[\frac{m(v)-1}{r-1}\big]}.
\end{align*}
\end{prop}

For an unlabeled topology $T$, we write $N(T)$ to tabulate labeled histories for a labeled topology obtained by an arbitrary labeling of the leaves of $T$. We seek to characterize all unlabeled topologies
$T^*$ such that 
\begin{equation}
\label{eq:argmax}
T^* \in \arg \max_{\{T:|T|=n\}} N(T),
\end{equation}
where the maximization is over the set of ($r$-furcating) unlabeled topologies with $n$ leaves. Using the monotonicity of the logarithm, the maximization of $N(T)$ has the following equivalence:
\begin{align}
    \max_{\{T : |T| = n\}} N(T) &\iff \max_{\{T : |T| =n\}} \frac{\big(\frac{n-1}{r-1}\big)!}{\prod_{v \in V^0(T)}\big[\frac{m(v)-1}{r-1}\big]} \nonumber \\
    &\iff \min_{\{T:|T|=n\}} \prod_{v \in V^0(T)}\bigg[\frac{m(v)-1}{r-1}\bigg] \nonumber \\
    &\iff \min_{\{T:|T|=n\}} \sum_{v \in V^0(T)} \log \big[ m(v) - 1 \big],
\label{eq:min}
\end{align}
noting that we need not consider terms that only involve the fixed constants $n$ and $r$. With our minimization characterized in eq.~\ref{eq:min}, we now present results from majorization theory, which will be used in our proofs. 


\section{Majorization lemmas}
\label{app:majorization}

We continue with definitions and lemmas from majorization theory, based on Marshall et al.~\cite{Marshall11}.

\begin{defi} [\cite{Marshall11}, Definition 1.A.1, p.~8] \label{def:majorize}
    For $x, y \in \mathbb{R}^n$, with the elements of $x, y$ written in non-increasing order, $x_i \geq x_{i+1}$ and $y_i \geq y_{i+1}$ for all $i \in \{1, 2, \ldots, n-1\}$, we say that $y$ \emph{majorizes} $x$, $x \prec y$, if
    \begin{enumerate}[label = (\roman*)]
        \item $\sum_{i=1}^k x_i \leq \sum_{i=1}^k y_i$ for all $k$, $1 \leq k \leq n-1$, and
        \item $\sum_{i=1}^n x_i = \sum_{i=1}^n y_i$.
    \end{enumerate}
\end{defi}
\begin{defi} [\cite{Marshall11}, Definition 3.A.1, p.~80] \label{def:Schur_convex}
    A real-valued function $\phi$ defined on a set $\mathcal{A} \subset \mathbb{R}^n$ is said to be \emph{Schur-convex} on $\mathcal{A}$ if for all $x, y \in A$
$$x \prec y \implies \phi(x) \leq \phi(y).$$
If, in addition, $\phi(x) < \phi(y)$ whenever $x \prec y$ but $x$ is not a permutation of $y$, then $\phi$ is said to be \emph{strictly} Schur-convex on $\mathcal{A}$.
 \end{defi}
\noindent
We will need a result about strict Schur-convexity of a sum of strictly convex functions in individual variables.
\begin{prop} [\cite{Marshall11}, 3.C.1.a, p.~92] \label{prop:strict_Schur}
    Let $I \subset \mathbb{R}$ be an interval, and for $x \in \mathbb{R}^n$, let $\phi(x) = \sum_{i=1}^n g(x_i)$, where $g:I \to \mathbb{R}$. If $g$ is strictly convex on $I$, then $\phi$ is strictly Schur-convex on $I^n$.
\end{prop}
We will also need the concept of \emph{weak supermajorization}. 
\begin{defi} [\cite{Marshall11}, Definition 1.A.2, p.~12] \label{def:supermajorization}
     For $x, y \in \mathbb{R}^n$, with the elements of $x, y$ written in non-decreasing order, 
$x_i \leq x_{i+1}$ and $y_i \leq y_{i+1}$ for all $i \in \{1, 2 \ldots, n-1\}$, we say that $y$ \emph{weakly supermajorizes} $x$, $x \prec^w y$, if for all $k$, $1 \leq k \leq n$,
$\sum_{i=1}^k y_i \leq \sum_{i=1}^k x_i$. 
\end{defi}
\noindent We note that majorization is commonly defined using non-increasing vectors and weak supermajorization is defined using non-decreasing vectors. 

Next, we will need a definition of ``increasing'' and ``decreasing'' for a multivariate function. 

\begin{defi}[\cite{Marshall11}, p.~637] \label{def:multi_increasing_1}
    For $x, y \in \mathbb{R}^n$, write $x \leq y$ if $x_i \leq y_i$, $i = 1, 2, \ldots, n$. A multivariate function $\phi:\mathbb{R}^n \to \mathbb{R}$ is said to be \emph{increasing} if 
    $$x \leq y \implies \phi(x) \leq \phi(y).$$
    Further, if
    $$x \leq y \text{ and } x \not = y \implies \phi(x) < \phi(y),$$
    then $\phi$ is said to be \emph{strictly} increasing. If $-\phi$ is (strictly) increasing, then $\phi$ is said to be (strictly) decreasing.
\end{defi}

\noindent We present a lemma about supermajorization when a Schur-convex function is applied.
\begin{lemma}[\cite{Marshall11}, 3.A.8.a, p.~87] \label{lemma:Schur_weakly_maj}
Let $\phi$ be a real-valued function defined on a set $\mathcal{A} \subset \mathbb{R}^n$. The following statement holds if and only if $\phi$ is strictly decreasing and strictly Schur-convex on $\mathcal{A}$:
$$x \prec^w y \text{ on } \mathcal{A} \text{ and } x \text{ is not a permutation of } y \implies \phi(x) < \phi(y).$$
\end{lemma}
\noindent We next present the main result of the section, a lemma that we will use in Section \ref{app:completing} in our maximization of the number of labeled histories over the set of unlabeled topologies.
\begin{lemma} \label{lemma:log_sum}
Let $x, y \in \mathbb{R}_{>0}^n$. If $x \prec^w y$ and $x$ is not a permutation of $y$, then $\sum_{i=1}^n \log(y_i) < \sum_{i=1}^n \log(x_i)$.
\end{lemma}
\begin{proof}
    Consider the function $\phi:\mathbb{R}^n \to \mathbb{R}$ defined by $\phi(z) = \sum_{i=1}^n -\log(z_i)$. First, we show that $\phi(z)$ is strictly decreasing and strictly Schur-convex on $\mathcal{A} = (0, \infty)^n \subset \mathbb{R}^n$.

    Consider $x, y \in \mathbb{R}^n_{> 0}$, where $x \leq y$ as in Definition \ref{def:multi_increasing_1} and $x \not = y$. We show that $\phi(x)$ is strictly decreasing by showing that $-\phi(x)$ is strictly increasing. For each $i$, because $x_i \leq y_i$, we have $\log(x_i) \leq \log(y_i)$, as $f(x) = \log(x)$ is strictly increasing for all $x \in (0, \infty)$. Because $x \not = y$, there must exist some $j$ such that $x_j < y_j$. For that $j$, using the strictly increasing property of $f(x) = \log(x)$, we have $\log(x_j) < \log(y_j)$. Thus, 
    $$-\phi(x) = \sum_{i=1}^n \log(x_i) < \sum_{i=1}^n \log(y_i)= -\phi(y),$$
    so that $\phi(z)$ is strictly decreasing, as needed.

    Next, $f(x) = - \log(x)$ is strictly convex on $(0, \infty)$, and therefore, by Proposition~\ref{prop:strict_Schur}, $\phi(z) = \sum_{i=1}^n - \log(z_i)$ is strictly Schur-convex on $\mathbb{R}_{>0}^n$.

    Let $\mathcal{A} = \mathbb{R}_{>0}^n$. Because $\phi$ is strictly decreasing and strictly Schur-convex on $\mathcal{A}$, if $x \prec^w y$ and $x$ is not a permutation of $y$, then by Lemma~\ref{lemma:Schur_weakly_maj},
   $$\sum_{i=1}^n -\log(x_i) < \sum_{i=1}^n -\log(y_i),$$
    completing the proof.
\end{proof}

\noindent With the necessary results from majorization theory, we proceed to results from information theory. In particular, we begin with the Huffman algorithm.

\section{The Huffman algorithm and Huffman trees}
\label{app:huffman}

\subsection{The Huffman algorithm}

Our next step in obtaining the maximally probable $r$-furcating unlabeled topology is to borrow results concerning Huffman trees in information theory. Following \cite{Glassey76}, let $\{v_1, v_2, \ldots, v_w\} \cup \{\ell_1, \ell_2, \ldots, \ell_n\}$ be the node set of a rooted tree $T$, where the $v_i$ are internal nodes, with $v_w$ as the root, and the $\ell_i$ are leaves. Let $\vec{c} = (c_1, c_2, \ldots, c_n)$ be a fixed nonnegative real $n$-vector of weights, where weight $c_i$ corresponds to leaf $\ell_i$. 

In the \emph{Huffman algorithm}~\cite[Section 5.6, p.~92]{Cover05}, given leaves $\{\ell_1, \ell_2, \ldots, \ell_n\}$ with associated weights $(c_1, c_2, \ldots, c_n)$, for fixed $r \geq 2$, a tree is constructed by a greedy algorithm that successively merges the $r$ nodes of minimal weight. Each newly constructed internal node has a weight equal to the sum of the weights of its children. If multiple minimal-weight sets of $r$ nodes exist, then one is chosen arbitrarily. 

\begin{figure}
    \centering
    \includegraphics[width=0.73\linewidth]{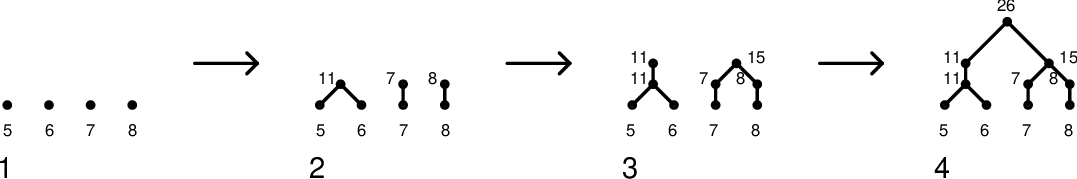}
    \vspace{-.25cm}
    \caption{The construction of the bifurcating $H$-tree for weight vector $\sigma=(5, 6, 7, 8)$. In panel 1, the leaves of weights $5$ and $6$ are merged to produce an internal node of weight 11. We are left to choose among nodes with weights $(7, 8, 11)$. In panel 2, the leaves of weights 7 and 8 are merged to produce an internal node of weight 15. Finally, in panel 3, the two internal nodes of weight 11 and 15 are merged to produce the root, with weight 26. The $H$-tree appears in panel 4. In the notation of the merge operator, $M_2(\sigma) = (7, 8, 11)$, $M^2_2(\sigma) = (11, 15)$, and $M^3_2(\sigma) = (26)$, and the weight sequence is $(11,15,26)$.}
    \label{fig:Huff_construct}
\end{figure}

A tree constructed by this algorithm is an \emph{$H$-tree}; an example appears in Figure \ref{fig:Huff_construct}. Because of the arbitrary choices that can be made at stages at which multiple sets tie for the minimal weight, the tree built by the Huffman algorithm is not necessarily unique. For example, Figure \ref{fig:distinct_trees_same_weight} shows two distinct unlabeled topologies constructed by the Huffman algorithm with a shared weight vector. In the context of finding maximally probable $r$-furcating unlabeled topologies on $n$ leaves, we investigate $r$-furcating $H$-trees with the weight vector $\vec{1}$. In this setting, the weight of an internal node, $v$, is its number of descendant leaves, $m(v)$.

First, we show that the $r$-furcating $H$-tree on $n$ leaves with the weight vector $\vec{1}$ is unique (Lemma~\ref{lemma:unique}), and we characterize its shape (Proposition~\ref{prop:h_tree_shape}). In Section \ref{app:completing}, we show that this unique $r$-furcating $H$-tree is the desired maximally probable $r$-furcating unlabeled topology on $n$ leaves.

\begin{figure}
    \centering
    \includegraphics[width=0.7\linewidth]{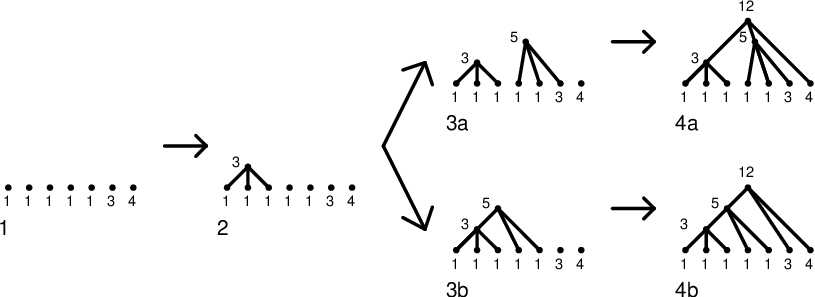}
    \vspace{-.25cm}
    \caption{Two distinct unlabeled topologies for trifurcating $H$-trees with weight vector $\sigma=(1, 1, 1, 1, 1, 3, 4)$. In panel 1, three nodes of weight 1 are merged to produce a node of weight 3. In panel 2, either the leaves of weights 1, 1, and 3 can be selected to produce the tree seen in panel 3a, or the newly produced node of weight 3 and the two leaves of weight 1 can be selected to produce the tree in panel 3b. The $H$-trees appear in panels 4a and 4b. In either case, $M_3(\sigma)=(1,1,3,3,4)$, $M_3^2(\sigma) = (3, 4, 5)$, and $M_3^3(\sigma)=(12)$, and the weight sequence is $(3, 5, 12)$.}
    \label{fig:distinct_trees_same_weight}
\end{figure}

To show that the $r$-furcating $H$-tree on $n$ leaves is unique, we first define the \emph{$r$-merge operator}, $M_r$, which generalizes the binary merge operator in Section 4 of \cite{Glassey76}. Let $\Sigma$ be the set of finite non-decreasing sequences of positive integers. $M_r: \Sigma \to \Sigma$ operates on a sequence $\sigma \in \Sigma$ of length $r$ or greater by deleting the first $r$ elements of $\sigma$ and inserting their sum into the resulting sequence at the appropriate place. As an example, if $r=3$ and $\sigma = (1, 1, 3, 4, 6, 7)$, then $M_3(\sigma) = (4, 5, 6, 7)$. The three smallest elements in $\sigma$ are $1, 1, 3$, and their sum of 5 is placed among the remaining elements so that the sequence remains non-decreasing.

The notation $M_r^i(\sigma)$ represents the application of $i \geq 0$ successive $r$ merge operations on sequence $\sigma$, where for convenience $M_r^1(\sigma)$ is usually written $M_r(\sigma)$, and $M_r^0(\sigma)=\sigma$. Beginning with a sequence of length $n = w(r-1) + 1$, $w \geq 1$, sequence $M_r^w(\sigma)$ has length $1$, as each merge operation reduces the sequence length by $r-1$. The operation preserves sums of elements: the elements of $M_r^i(\sigma)$ and $\sigma$ have the same sums.

In the context of the Huffman algorithm, $\sigma$ represents the weight vector, the list of weights of the starting sequence. An application of the $r$-merge operator, $M_r(\sigma)$, corresponds to a step of the Huffman algorithm: the $r$ nodes of smallest weight are merged into a new node, which has a weight corresponding to the sum of these $r$ smallest weights (Figure \ref{fig:Huff_construct}). 

\subsection{Normal sequences}

Following \cite{Glassey76}, we now define a \emph{normal sequence} in the setting of $r$-furcation.
\begin{defi} \label{def:normal_seq}
    A finite non-decreasing sequence of positive integers $\sigma$ is \emph{normal} if there exists a nonnegative integer $k$ such that:
\begin{enumerate}
    \item Each element of $\sigma$ lies in $[r^k, r^{k+1}]$, and 
    \item At most one element of $\sigma$ is not a power of $r$ (where $1=r^0$ is regarded as a power of $r$).
\end{enumerate}
\end{defi}
Crucially, if $\sigma$ is a normal sequence and $k$ is its associated nonnegative integer, then every element of $\sigma$ has one of three possible values: $r^k, r^{k+1}, r^k + b$, where $0 < b < r^{k+1} - r^k$, and at most one element is $r^k+b$. 

We now generalize Lemma 3 of \cite{Glassey76} about properties of normal sequences to $r$-furcation.
\begin{lemma} 
\label{lemma:merge_op}
The following properties hold concerning normal sequences: 
    \begin{enumerate}
        \item For each $n \geq 1$, the sequence consisting of $n$ $1'$s is normal.
        \item If $\sigma$ is normal, then $M_r(\sigma)$ is normal.
        \item If $\sigma = (\sigma_1, \sigma_2, \ldots, \sigma_r, \ldots)$ is normal, then at least $r-1$ of the first $r$ elements $\sigma_1, \sigma_2, \ldots, \sigma_r$ equal powers of $r$, and $\sigma_1 / \sigma_r \geq \frac{1}{r}$.
    \end{enumerate}
\end{lemma}
\begin{proof}
    Let $\sigma$ be a normal sequence of length $n$.
    \begin{enumerate}
        \item This result holds by Definition \ref{def:normal_seq}, as each element of the sequence of $n$ $1$'s is $r^0$, a power of $r$.
        \item By Definition \ref{def:normal_seq}, there exists $k \geq 0$ such that each $\sigma_i$ has one of three possible values: $r^k, r^{k+1}, r^k + b$, where $0 < b < r^{k+1} - r^k$. Further, there exists $\ell$, $1 \leq \ell \leq n$, such that $\sigma_i = r^k$ for $i \in \{1, 2, \ldots, \ell-1\}$, $\sigma_i = r^{k+1}$ for $i \in \{\ell+1, \ell+2, \ldots, n\}$, and $\sigma_\ell=r^k + b$, where $0 \leq b < r^{k+1} - r^k$. We have two cases:
        \begin{enumerate}
            \item Suppose $\ell > r$ or $\ell = r$ with $b=0$. Either way, $\sigma_i = r^k$ for $i \in \{1, 2, \ldots, r \}$. We then have $\sum_{i=1}^r \sigma_i = r^{k+1}$, which is a power of $r$. $M_r(\sigma)$ has at most one element that is not a power of $r$ because the merge operation replaces $r$ powers of $r$ (all $r^k$) with a different power of $r$ ($r^{k+1}$). In addition, all elements of $M_r(\sigma)$ lie in $[r^k, r^{k+1}]$ because the unmerged elements were in this range to start, and the new element $r^{k+1}$ is also in this range.
            \item Suppose $\ell < r$ or $\ell = r$ with $b \not = 0$. Either way, $\sigma_r > r^k$. With $\sum_{i=1}^r \sigma_i = \ell r^{k} + (r - \ell) r^{k+1} + b$, it is clear that $r^{k+1} < \sum_{i=1}^r \sigma_i < r^{k+2}$: for the left-hand side, each of $r$ terms in $\sum_{i=1}^r \sigma_i$ is at least $r^k$, and at least one, the $r^k+b$ term, strictly exceeds $r^k$; For the right-hand side,
            \begin{align*}
             \sum_{i=1}^r \sigma_i & < \ell r^k + (r-\ell)r^{k+1} + (r^{k+1} - r^k) = r^{k+2} - (r-1)(\ell-1)r^k \leq r^{k+2}.
            \end{align*}
            The unmerged elements all equal $r^{k+1}$, so all elements of $M_r(\sigma)$ lie in $[r^{k+1}, r^{k+2}]$. Note that $\sigma_\ell$ is the only element of $\sigma$ that is possibly not a power of $r$. Because $\sigma_\ell$ is used in the merging, it is incorporated in the new element of $M_r(\sigma)$, so at most one element of $M_r(\sigma)$ is not a power of $r$.
        \end{enumerate}
        \item By property 2 of Definition \ref{def:normal_seq}, at most one element of $\sigma$ is not a power of $r$. Hence, for the subset of elements, $\sigma_1, \sigma_2, \ldots, \sigma_r$, at most one element is not a power of $r$, and at least $r-1$ elements are equal powers of $r$. For the second part of the statement, we have three cases, in all of which $\sigma_1/\sigma_r \geq \frac{1}{r}$:
        \begin{enumerate}
            \item $(\sigma_1, \sigma_r) = (r^k, r^k)$. Then $\sigma_1/\sigma_r = 1 \geq \frac{1}{r}$.
            \item $(\sigma_1, \sigma_r) = (r^k, r^{k+1})$. Then $\sigma_1/\sigma_r = \frac{1}{r} \geq \frac{1}{r}$. 
            \item $(\sigma_1, \sigma_r) = (r^k, r^k + b)$, $0 < b < r^{k+1}-r^k$. Then $\sigma_1/\sigma_r > r^k / r^{k+1} = 1/r$.
        \end{enumerate}
    \end{enumerate}
\end{proof}

An important application of properties 1 and 2 of Lemma~\ref{lemma:merge_op} that we use in the remaining proofs is that $M_r^i(\vec{1})$ is a normal sequence. 

\subsection{Uniqueness of the $H$-tree for a uniform weight vector}

We now study the properties of the $r$-merge operator and the trees it produces, toward proving the uniqueness of the $H$-tree for a uniform weight vector. We begin by showing that if a normal sequence $M_r^i(\vec{1})$ has a term $r^{k+1}$, then that term must have been produced by merging $r$ copies of $r^{k}$. First, we state a lemma from \cite{Dickey25}.

\begin{lemma} [\cite{Dickey25}, Lemma 10]
\label{lem:10decomp}
Choose positive integers $q,r \geq 2$. Let $k_0=0$ and $k_p = (q-1)r^{p-1}$ for integers $p \geq 1$, and let $j_p = k_{p+1} - k_p$ for $p \geq 0$. Each integer $n \geq q$ has a unique decomposition specified by integers $(p,s,b)$ as $n=s k_p + (r-s)k_{p+1} + b$, where $1 \leq s \leq r$, $p \geq 0$, and $0 \leq b < j_p$. 
\end{lemma}

In the proof of this lemma, taking $q=2$, it is shown in \cite{Dickey25} that given $n$, the decomposition has $p = \lfloor \log_r n \rfloor$, $s = \lfloor (r^{p+1} - n)/(r^p - r^{p-1}) \rfloor$, and $b = n - [sr^{p-1} + (r-s)r^p]$. With this lemma and unique decomposition, we can now prove that an $r^{k+1}$ term in a normal sequence must arise from the merging of $r$ copies of $r^k$.

\begin{lemma} \label{lemma:r_k_only}
    Fix $r \geq 2$, and let $n = w(r-1) + 1$ for a nonnegative integer $w$. Let $\vec{1}$ be an $n$-vector, and let $\sigma = M_r^i(\vec{1})$ for some $i \geq 1$. For $0 \leq k \leq \lfloor \log_r n \rfloor - 1$, if $M_r^i(\sigma)$ contains a term of the form $r^{k+1}$, then it was produced by merging $r$ copies of $r^{k}$.
\end{lemma}

\begin{proof}
    Recall by property 1 of Lemma~\ref{lemma:merge_op} that $\vec{1}$ is a normal sequence. By property 2 of Lemma~\ref{lemma:merge_op}, sequence $\sigma = M_r^i(\vec{1})$ is also normal. Because $\sigma$ is normal, using Definition \ref{def:normal_seq}, a term of the form $r^{k+1}$ can only be produced by merging terms of the form $r^p$, $r^{p+1}$, and $r^p + b$, where $0 \leq b < r^{p+1} - r^p$. For $q = 2$ and our fixed $r$, Lemma~\ref{lem:10decomp} states that $r^{k+1}$ has a unique decomposition as
    $$r^{k+1} = s \cdot r^{p-1} + (r-s) \cdot r^p + b,$$
    where $1 \leq s \leq r$, $p \geq 0$, and $0 \leq b < r^{p+1}-r^p$. Take $(p,s,b) = (k+1,r,0)$; the decomposition 
    $$r^{k+1} = r \cdot r^{k} + (r-r) \cdot r^{k+1} + 0,$$
    must be unique. Hence, the only way to produce an $r^{k+1}$ term is by merging $r$ copies of $r^k$, as needed.
\end{proof}

We now connect the $r$-merge operator and its properties to the construction of trees and to their shapes. A fully symmetric $r$-furcating tree is an $r$-furcating tree in which, for each internal node, all $r$ subtrees have the same unlabeled shape. 

\begin{lemma} \label{lemma:powers_are_sym}
    Fix $r \geq 2$, and let $n = w(r-1) + 1$ for a nonnegative integer $w$. Let $S$ be an $r$-furcating $H$-tree on $n$ leaves constructed from the weight vector, $\vec{1}$, according to the Huffman algorithm. For each $v \in V^0(S)$ such that $m(v) = r^k$, $1 \leq k \leq \lfloor \log_r n \rfloor$, the tree rooted at $v$ is the fully symmetric $r$-furcating tree.
\end{lemma}

\begin{proof}
    We proceed by induction on $k$. In the base case of $k=1$, with the weight vector $\vec{1}$, a node $v$ with weight $r^1$ has $r$ children, each of weight $1$. Thus, $v$ is a (fully symmetric) broomstick node with $r$ leaves.

    For the inductive hypothesis, suppose that for all $k$, $1 \leq k \leq k_* < \lfloor \log_r n \rfloor$, trees that are rooted at nodes $v$ for which $m(v)=r^k$ are fully symmetric $r$-furcating trees. Consider a node $v$ for which $m(v)=r^{k_*+1}$. By Lemma~\ref{lemma:r_k_only}, terms with value $r^{k_*+1}$ in a weight vector obtained in the construction of an $H$-tree can only be produced by merging $r$ copies of $r^{k_*}$, each, by the inductive hypothesis, representing a fully symmetric $r$-furcating tree. Hence, merging $r$ values $r^{k_*}$ produces a fully symmetric $r$-furcating tree on $r^{k_*+1}$ leaves.
\end{proof}

\begin{corollary} \label{cor:perf_power_H_tree}
    For $r \geq 2$, the unique $r$-furcating $H$-tree on $n=r^k$ leaves, $k \geq 1$, constructed by the Huffman algorithm from the weight vector $\vec{1}$, is the fully symmetric $r$-furcating tree on $n$ leaves.
\end{corollary}
\begin{proof}
    Let $T$ be an $r$-furcating $H$-tree on $n=r^k$ leaves. The root, $v \in V^0(T)$, has $m(v) = r^k$. By Lemma~\ref{lemma:powers_are_sym}, the tree rooted at $v$ is the fully symmetric $r$-furcating tree, so that $T$ is this tree as well.
\end{proof}
Now, we show that although Huffman trees are not necessarily unique when starting with an arbitrary weight vector, with the weight vector $\vec{1}$, the $H$-tree on $n = w(r-1)+1$ leaves, $w \geq 1$, is unique.

\begin{lemma} \label{lemma:unique}
Fix $r \geq 2$, and let $n=w(r-1)+1$ for a nonnegative integer $w$. For a uniform weight vector on $n$ leaves, there exists a unique $r$-furcating $H$-tree shape.
\end{lemma}
\begin{proof}
    Recall that the Huffman algorithm can potentially produce multiple $H$-trees if there exists a step in which multiple node sets tie for the minimal weight and these sets correspond to different tree shapes. We show that if multiple sets tie for the minimal weight, then each choice among the sets produces the same tree shape. Consider the length-$n$ weight vector $\vec{1}$ and a positive integer $i$. For $M_r^i(\vec{1})$, there are four cases:
    \begin{enumerate}
        \item $M_r^i(\vec{1})$ selects $r$ identical elements from the previous vector $M_r^{i-1}(\vec{1})$, or the $r$th and $(r+1)$th elements of $M_r^{i-1}$ take different values. In either case, a single set has the minimal weight.

        \item $M_r^i(\vec{1})$ selects $r$ identical terms of the form $r^k$ from more than $r$ such terms. By Lemma~\ref{lemma:powers_are_sym}, all terms $r^k$ represent fully symmetric $r$-furcating trees. Hence, irrespective of the choice of which $r$ terms are selected, they produce a fully symmetric $r$-furcating tree on $r^{k+1}$ leaves.

        \item $M_r^i(\vec{1})$ selects only terms of the form $r^k$ and $r^{k+1}$. Because at least one term $r^{k+1}$ term is selected, \emph{all} $r^k$ terms are selected. A choice might exist among $r^{k+1}$ terms; however, by Lemma~\ref{lemma:powers_are_sym}, all represent fully symmetric $r$-furcating trees, so that any set of $r^{k+1}$ terms results in the same tree shape.

        \item $M_r^i(\vec{1})$ selects terms of the form $r^k + b$ for $0 < b < r^{k+1}-r^k$. Because each sequence $M_r^i(\vec{1})$ is always normal by Lemma~\ref{lemma:merge_op}, at most one term has form $r^k + b$. Hence, the only possible choice is among terms of the form $r^{k+1}$. Again using Lemma~\ref{lemma:powers_are_sym}, each such choice leads to the same tree.
    \end{enumerate}
    We conclude that only one tree shape can be produced by the Huffman algorithm starting with $\vec{1}$.
\end{proof}
 
Now that we have shown that the $r$-furcating $H$-tree on $n$ leaves is unique, we characterize its shape. Note that a version of this proposition was demonstrated on p.~256 of \cite{Glicksman65} in the setting of sequentially merging identical units of tape.
\begin{prop} 
\label{prop:h_tree_shape}
    For $r \geq 2$ and $w \geq 1$. With a uniform weight vector, the unique $r$-furcating $H$-tree on $n=w(r-1) + 1$ leaves takes the form $U^*_n = (s-1)U_{r^{p-1}}^* \oplus U_{r^{p-1}+b}^* \oplus (r-s)U_{r^p}^*$, where 
    \begin{align*}
        p &= \lfloor \log_r n \rfloor, \\
        s &= \bigg\lfloor \frac{r^{p+1}-n}{r^p - r^{p-1}} \bigg\rfloor,\\
        b &= n-\big[sr^{p-1}+(r-s)r^p \big]. 
    \end{align*}
\end{prop}
\begin{proof}
    Choose $w \geq 1$ and let $n = w(r-1) + 1$. Let $\vec{1}$ be a length-$n$ uniform weight vector. By Lemma~\ref{lem:10decomp}, $n$ has a unique decomposition of the form $n = sr^{p-1} + (r-s)r^p + b$, where $p \geq 1$, $1 \leq s \leq r$, and $0 \leq b < r^{p}-r^{p-1}$. In this unique decomposition, $(p,s,b)$ follow the statement of the proposition.

    Consider $M_r^{w-1}(\vec{1})$, which is a sequence of length $r$. Because $M_r^{w-1} (\vec{1})$ is normal by Lemma~\ref{lemma:merge_op}, its $r$ elements all lie in $[r^k,r^{k+1}]$ for some $k$, and at most one element is not a power of $r$. The sum of these $r$ elements is $n$. Using our unique decomposition of $n$ in Lemma~\ref{lem:10decomp}, sequence $M_r^{w-1}(\vec{1})$ must have $s-1$ elements $r^{p-1}$, an element $r^{p-1} + b$, and $r-s$ elements $r^p$, where $(p,s,b)$ follow the statement of the proposition.

    By Lemma~\ref{lemma:powers_are_sym}, terms of the form $r^{p-1}$ and $r^p$ in $M_r^{w-1}(\vec{1})$ correspond to fully symmetric $r$-furcating trees, which, by Corollary~\ref{cor:perf_power_H_tree}, are the $r$-furcating $H$-trees on $r^{p-1}$ and $r^p$ leaves. 

    It remains to show that the subtree for $M_r^{w-1} (\vec{1})$ on $r^{p-1} + b$ leaves is the $r$-furcating $H$-tree on $r^{p-1} + b$ leaves. By Lemma~\ref{lem:10decomp}, $r^{p-1}+b$ has a unique decomposition as the sum of terms $r^{p-2}, r^{p-1}, r^{p-2} + b^{'}$, where $0 \leq b^{'} < r^{p-1}-r^{p-2}$. By Lemma~\ref{lemma:powers_are_sym}, the $r^{p-2}$ and $r^{p-1}$ terms correspond to fully symmetric $r$-furcating trees. Continuing with the term $r^{p-2} + b^{'}$, by Lemma~\ref{lem:10decomp}, $r^{p-2} + b^{'}$ has a unique decomposition as the sum of terms $r^{p-3}, r^{p-2}, r^{p-3}+b^{''}$, where $0 \leq b^{''} \leq r^{p-2}-r^{p-3}$. By Lemma~\ref{lemma:powers_are_sym}, the $r^{p-3}$ and $r^{p-2}$ terms correspond to fully symmetric $r$-furcating trees. 

    Continue decomposing the subtree with the additive term $b^{'\cdots'}$, until the inequality satisfied by $b^{'\cdots '}$ is $0 \leq b < r-1$. When this stage is reached, $b^{'\cdots'}$ must equal 0 or 1, as nodes $v$ of an $r$-furcating tree cannot have a number of descendant leaves $m(v)$ strictly between 1 and $r$. Each subtree at this stage has a number of descendant leaves that is a power of $r$, so that all are fully symmetric $r$-furcating trees by Lemma~\ref{lemma:powers_are_sym}. Hence, the subtree for $M_r^{w-1} (\vec{1})$ on $r^{p-1}+b$ leaves and the $H$-tree on $r^{p-1}+b$ leaves have identical decompositions at each step, so the subtree on $r^{p-1} + b$ leaves is the $r$-furcating $H$-tree on $r^{p-1} + b$ leaves. Using the uniqueness of the $r$-furcating $H$-tree in Lemma~\ref{lemma:unique}, we have uniquely established the shape of this tree.
\end{proof}

\section{The maximally probable $r$-furcating tree}
\label{app:completing}

With the unique $r$-furcating $H$-tree of Proposition~\ref{prop:h_tree_shape} in hand, we now show that it is maximally probable. For this demonstration, we use a second aspect of the Huffman algorithm, the \emph{weight sequence}, which connects to majorization theory. An $r$-furcating unlabeled topology $T$ with a weight vector for its leaves has a weight sequence $\omega(T) = \big(\omega(T_1), \omega(T_2), \ldots, \omega(T_{w}) \big)$, a vector of the internal node weights in non-decreasing order. 

The weight sequence corresponds to the mergings of the Huffman algorithm; $\omega(T_i)$ corresponds to the sum of the weights of the $r$ nodes merged in step $i$. 
Note that although the $H$-tree is unique for the uniform weight vector, in general, given an initial weight vector \emph{and} a weight sequence, the $H$-tree is not necessarily unique. For example, in Figure \ref{fig:distinct_trees_same_weight}, the two distinct tree topologies beginning from the same weight vector $(1,1,1,1,1,3,4)$ both have the weight sequence $(3, 5, 12)$.

We now state a property of the weight sequence for $r$-furcating $H$-trees. 



\begin{lemma} [\cite{Glassey76}, Lemma 1, p.~371] \label{lemma:hu_majorization}
    Fix a weight vector $\vec{c}$ for $n$ leaves, $n=w(r-1)+1$, $w \geq 1$. Let $S$ be an $r$-furcating unlabeled topology with $n$ leaves. $S$ is an $H$-tree associated with $\vec{c}$ if and only if for every $r$-furcating unlabeled topology $T$ associated with $\vec{c}$, $\omega(T) \prec^w \omega(S)$.
\end{lemma}
As a consequence of this lemma, if multiple distinct $H$-trees were to be associated with the weight vector $\vec{c}$, then they must have the same weight sequence. With all the necessary lemmas demonstrated, we finally prove our main result.

\begin{theorem} \label{thm:r_furcating_max_shape}
    Fix $r\geq 2$ and $w \geq 1$. The unique $r$-furcating unlabeled topology whose labelings maximize the number of labeled histories among labeled topologies with $n=w(r-1)+1$ leaves takes the form $U^*_n = (s-1)U_{r^{p-1}}^* \oplus U_{r^{p-1}+b}^* \oplus (r-s)U_{r^p}^*$, where 
    \begin{align*}
        p &= \lfloor \log_r n \rfloor, \\
        s &= \bigg\lfloor \frac{r^{p+1}-n}{r^p - r^{p-1}} \bigg\rfloor,\\
        b &= n-\big[sr^{p-1}+(r-s)r^p \big]. 
    \end{align*}
\end{theorem}
\begin{proof}
    Recall from eq.~\ref{eq:min} that we are looking for an $r$-furcating unlabeled topology on $n$ leaves that minimizes
    \begin{equation}
    \label{eq:objective}
        \sum_{v \in V^0(T)} \log \big[ m(v) - 1 \big].
    \end{equation}
    We wish to show that only one unlabeled topology achieves the minimal value. We claim that this unlabeled topology is the $r$-furcating $H$-tree on $n$ leaves constructed from the weight vector $\vec{1}$.

    By Proposition $\ref{prop:h_tree_shape}$, for the weight vector $\vec{1}$, the unique $r$-furcating $H$-tree, $S$, has the shape defined in the statement of the theorem. Consider another $r$-furcating unlabeled topology on $n$ leaves, $T$, distinct from $S$. Because $S$ is the unique $r$-furcating $H$-tree, by Lemma~\ref{lemma:hu_majorization}, $\omega(T) \prec^w \omega(S)$ and $\omega(T) \not = \omega(S)$. 

    Recall that elements of $\omega(S)$ and $\omega(T)$ correspond to the weights of internal nodes, which for the weight vector $\vec{1}$ corresponds to $m(v)$ for internal node $v \in V^0(S)$ or $v \in V^0(T).$ 

    The objective function in eq.~\ref{eq:objective} has terms $m(v) - 1$, so we construct $\omega'(S) = \omega(S) - \vec{1}$ and $\omega'(T) = \omega(T) - \vec{1}$. Because $\omega_i(S), \omega_i(T) \geq r$ for all $i \in \{1, 2, \ldots, w\}$, we know that $\omega'(S), \omega'(T) \in \mathbb{R}_{>0}^w$. In addition, we still have $\omega'(T) \prec^w \omega'(S)$ and $\omega'(S) \not = \omega'(T)$. Each element of $\omega'(S)$ represents $m(v) - 1$ for some $v \in V^0(S)$, and each element of $\omega'(T)$ represents $m(v)-1$ for some $v \in V^0(T)$. 

    Because $\omega'(S), \omega'(T) \in \mathbb{R}^w_{>0}$, the non--decreasing $\omega'(S)$ is not a permutation of the non-decreasing $\omega'(T)$, and $\omega'(T) \prec^w \omega'(S)$, by Lemma~\ref{lemma:log_sum}, we have that
    $$\sum_{i=1}^w \log \big[ \omega'_i(S) \big] < \sum_{i=1}^w \log \big[ \omega'_i(T) \big],$$
    or equivalently, 
    $$\sum_{v \in V^0(S)} \log \big[ m(v) - 1 \big] < \sum_{v \in V^0(T)} \log \big[ m(v) - 1 \big].$$
    Because the argument holds for arbitrary $T \neq S$, we conclude that $S = U^*_n$ is the unique minimizer of our objective function in eq.~\ref{eq:objective}, as desired.
\end{proof}

We have obtained the unique $r$-furcating unlabeled topology that maximizes the number of labeled histories for its associated labeled topologies. In doing so, we have verified Conjecture 13 of \cite{Dickey25}. The $r=2$ case of the proof provides a new proof of Theorem~\ref{thm: hammersley}, the result of \cite{Hammersley74} characterizing the bifurcating unlabeled topology whose labelings maximize the number of labeled histories. 

\section{Extensions to simultaneity}
\label{app:withsim}


With the maximally probable $r$-furcating unlabeled topology---without simultaneity---established in Theorem~\ref{thm:r_furcating_max_shape}, we present a conjecture for the maximally probable $r$-furcating unlabeled topology \emph{with} simultaneity, generalizing Conjecture 14 for bifurcation in \cite{Dickey26}. 

Following \cite{Dickey25, Dickey26, King23}, \emph{tie-permitting} labeled histories allow certain internal nodes to possess the same time; it is possible for internal nodes $u$ and $v$ with $u \neq v$ to possess the same time in some tie-permitting labeled history if and only if $u$ and $v$ do not have an ancestor--descendant relationship. A time for at least one internal node is an \emph{event}; events are numbered backward in time from the leaves toward the root. 

Theorem 15 of \cite{Dickey25} described a method of counting tie-permitting labeled histories for a labeled topology, or equivalently, for an arbitrary labeling of an unlabeled topology. In Table \ref{tab:tri_opt}, for odd $n$, $3 \leq n \leq 29$, we compute the maximal number of tie-permitting labeled histories across all trifurcating unlabeled topologies with $n$ leaves and $z$ events, for each $z$ in $\lceil \log_3 n \rceil \leq z \leq (n-1)/2$. For $n=1$, one labeled history occurs for $z=0$, and 0 labeled histories occur for all other $z$. As is true in analogous computations in the bifurcating case (Tables 3 and 4 of \cite{Dickey26}), all $(n,z)$ pairs in the table are achieved by the unlabeled topology in Theorem~\ref{thm:r_furcating_max_shape}, but the maximizing unlabeled topology is not necessarily unique. For fixed $n$ in the table, this unlabeled topology also maximizes the sum of the number of tie-permitting labeled histories across values of $z$.
\begin{conjecture} 
\label{conj:max_r_sim}
Consider the set of rooted $r$-furcating unlabeled topologies with $n$ leaves, permitting simultaneous $r$-furcations. 
\begin{enumerate}[label = (\roman*)]
    \item The unlabeled topology whose labelings have the largest number of tie-permitting labeled histories takes the form in Theorem~\ref{thm:r_furcating_max_shape}. This topology is unique in having the maximal value.
    \item Further, this same unlabeled topology has the largest number of tie-permitting labeled histories with exactly $z$ events, $\lceil \log_r n \rceil \leq z \leq (n-1)/(r-1)$. This topology is not necessarily unique in having the maximal value.
\end{enumerate}
\end{conjecture}
Generalizing an observation seen in the bifurcating case \cite{Dickey26}, for $r$-furcating trees, $r \geq 2$, the unlabeled topology in Theorem~\ref{thm:r_furcating_max_shape} is the unique maximizing unlabeled topology for $z = (n-1)/(r-1)$---as this maximal value of $z$ given $n$ and $r$ permits only non-simultaneous internal nodes---so that if part (ii) of the conjecture can be demonstrated, then part (i) follows.

\begin{figure}
    \centering
    \includegraphics[width=0.8\linewidth]{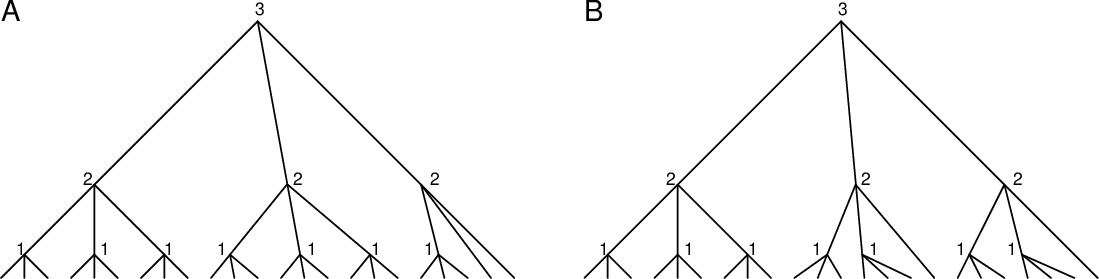}
    \caption{The two rooted trifurcating unlabeled topologies whose labelings produce the maximal number of tie-permitting labeled histories for $(n,z) = (23, 3)$. Both unlabeled topologies produce 1 tie-permitting labeled history. Internal nodes are annotated by the events to which they are assigned. (A) The unlabeled topology in Theorem~\ref{thm:r_furcating_max_shape}. (B) An alternative unlabeled topology.}
    \label{fig:non_unique_tri}
\end{figure}

As stated in part (ii), the maximum at fixed $(n,z)$ does not always occur for a unique unlabeled topology. For example, the two trifurcating tree shapes in Figure \ref{fig:non_unique_tri} produce the maximal number of tie-permitting labeled histories for $(n,z) = (23, 3)$, namely 1.

This pattern of non-uniqueness generalizes. For each $k \geq 3$, for $(n, z) = (2 \cdot 3^{k-1} + 3^{k-2} + 2, k)$, non-uniqueness occurs with two unlabeled topologies of similar structure. One of the unlabeled topologies---call it $T_1$---has subtrees $L_1, C_1, R_1$ of size $(|L_1|, |C_1|, |R_1|)=(3^{k-1}, 3^{k-1}, 3^{k-2} + 2)$. The other, $T_2$, has subtrees $L_2, C_2, R_2$ of size $(|L_2|, |C_2|, |R_2|)=(3^{k-1}, 3^{k-1}-2, 3^{k-2} + 4)$. In the two unlabeled topologies, subtrees $L_1, C_1, R_1, L_2, C_2, R_2$ have the unlabeled topologies in Theorem~\ref{thm:r_furcating_max_shape}. $T_1$ and $T_2$ both produce 1 tie-permitting labeled history, as all internal nodes lie on a length-$k$ path between leaves and the root.

In fact, we can generalize this tree structure to $r$-furcating trees, $r \geq 3$. For each $k \geq 3$, for $(n, z) = \big((r-1) \cdot r^{k-1} + r^{k-2} + (r-1) , k \big)$, non-uniqueness occurs with unlabeled topologies of similar structure. We construct just two of these unlabeled topologies. $T_1$ has subtrees $S^1_1, S^2_1, \ldots, S^r_1$ of size $(|S^1_1|, |S^2_1|, \ldots, |S^{r-1}_{1}|, |S_1^r|)= \big(r^{k-1}, r^{k-1}, \ldots, r^{k-1}, r^{k-2}+(r-1)\big)$. $T_2$ has subtrees $S^1_2, S^2_2, \ldots, S^r_2$ of size $(|S^1_2|, |S^2_2|, \ldots, |S^{r-1}_{2}|, |S_2^r|)=\big(r^{k-1}, r^{k-1}, \ldots, r^{k-1}, r^{k-1} - (r-1), r^{k-2}+ 2(r-1)\big)$. In the two unlabeled topologies, all subtrees have the unlabeled topologies in Theorem~\ref{thm:r_furcating_max_shape}. Unlabeled topologies $T_1$ and $T_2$ both produce 1 tie-permitting labeled history because all internal nodes lie on a length-$k$ path between leaves and the root.

\begin{table}[tb]
    \centering
    \footnotesize
    \begin{tabular}{|c|c|c|c|c|c|c|c|c|c|c|c|c|c|c|}
        \hline
          & \multicolumn{14}{c|}{Number of leaves, $n$}  \\ 
        \cline{2-15}
    $z$   & 3 & 5 & 7 &  9 & 11 &  13 &  15 &    17 &     19 &      21 &        23 &         25 &          27 &            29 \\ \hline
    1     & 1 & 1 & 1 &  1 &  0 &   0 &   0 &     0 &      0 &       0 &         0 &          0 &           0 &             0 \\ \hline
    2     & 0 & 0 & 1 &  6 &  0 &   0 &   0 &     0 &      0 &       0 &         0 &          0 &           0 &             0 \\ \hline
    3     & 0 & 0 & 2 &  6 &  4 &   4 &   4 &     2 &      2 &       2 &         1 &          1 &           1 &             0 \\ \hline
    4     & 0 & 0 & 0 &  6 & 15 &  33 &  69 &    75 &    129 &     237 &       240 &        402 &         726 &           324 \\ \hline
    5     & 0 & 0 & 0 &  0 & 12 &  68 & 276 &   552 &  1,488 &   4,224 &     6,810 &     16,530 &      43,746 &        38,880 \\ \hline
    6     & 0 & 0 & 0 &  0 &  0 &  40 & 390 & 1,470 &  6,250 &  26,490 &    63,540 &    213,320 &     774,000 &     1,072,360 \\ \hline
    7     & 0 & 0 & 0 &  0 &  0 &   0 & 180 & 1,620 & 11,820 &  76,680 &   271,170 &  1,248,450 &   6,075,900 &    12,061,785 \\ \hline
    8     & 0 & 0 & 0 &  0 &  0 &   0 &   0 &   630 & 10,290 & 112,140 &   604,800 &  3,886,260 &  25,424,280 &    70,014,882 \\ \hline
    9     & 0 & 0 & 0 &  0 &  0 &   0 &   0 &     0 &  3,360 &  80,640 &   730,800 &  6,879,600 &  61,923,960 &   235,479,636 \\ \hline
    10    & 0 & 0 & 0 &  0 &  0 &   0 &   0 &     0 &      0 &  22,680 &   453,600 &  6,955,200 &  90,720,000 &   485,318,736 \\ \hline
    11    & 0 & 0 & 0 &  0 &  0 &   0 &   0 &     0 &      0 &       0 &   113,400 &  3,742,200 &  78,813,000 &   622,588,680 \\ \hline
    12    & 0 & 0 & 0 &  0 &  0 &   0 &   0 &     0 &      0 &       0 &         0 &    831,600 &  37,422,000 &   485,155,440 \\ \hline
    13    & 0 & 0 & 0 &  0 &  0 &   0 &   0 &     0 &      0 &       0 &         0 &          0 &   7,484,400 &   210,311,640 \\ \hline
    14    & 0 & 0 & 0 &  0 &  0 &   0 &   0 &     0 &      0 &       0 &         0 &          0 &           0 &    38,918,880 \\ \hline
    Total & 1 & 1 & 4 & 19 & 31 & 145 & 919 & 4,349 & 37,029 & 356,733 & 1,738,361 & 17,617,292 & 210,188,713 & 1,545,806,523 \\
    \hline
\end{tabular}
    \caption{For $3 \leq n \leq 29$ leaves, the number of tie-permitting labeled histories for the maximally probable trifurcating unlabeled topology, allowing simultaneity. Columns correspond to the number of leaves, $n$, and rows to the number of events, $z$, $\lceil \log_3 n \rceil \leq z \leq (n-1)/2$. Each entry is found by computing the number of tie-permitting labeled histories, $E(T, z)$, using Proposition 11 of \cite{Dickey26}, where $r=3$, for all trifurcating unlabeled topologies on $n$ leaves, and taking the maximum. The maximizing unlabeled topology is, in all $(n,z)$ in the table, the unlabeled topology described by Theorem~\ref{thm:r_furcating_max_shape}, but this topology might not be the only maximizing unlabeled topology. The maximizing total, summing across rows, is unique. The $n=23$ column corresponds to row 5 of OEIS A122193, and the $n=27$ column corresponds to row 6 of OEIS A122193.}
    \label{tab:tri_opt}
\end{table}


\section{Discussion}
\label{app:discussion}

We have generalized the Harding--Hammersley--Grimmett result describing the shape of the maximally probable bifurcating unlabeled topology,  identifying the maximally probable \emph{$r$-furcating} unlabeled topology for $r \geq 2$ (Theorem \ref{thm:r_furcating_max_shape}). In completing this generalization, we have obtained a new proof of the characterization of the maximally probable unlabeled topology for the bifurcating case (Theorem \ref{thm: hammersley}). Our theorem for general $r$ verifies a conjecture from \cite{Dickey25}.

In the bifurcating case, maximally probable unlabeled topologies have been used in the mathematical analysis of likely outcomes of evolutionary processes under the Yule--Harding model for evolutionary descent~\cite{Degnan06}. These topologies have appeared as maximizers for other functions computed from bifurcating trees~\cite{DisantoAndRosenberg15}, and they have also inspired the development of analogous concepts of ``maximally probable'' for other classes of bifurcating trees~\cite{DegnanEtAl12:tcbb, DisantoEtAl19}. By obtaining the maximally probable $r$-furcating unlabeled topologies, our work here suggests that the applications of maximally probable trees in the bifurcating case might be possible to extend to settings with $r$-furcation. 

In \cite{Dickey25}, our conjecture for the shape of the maximally probable unlabeled topology was inspired in part by the work of \cite{Batty82}, which considered recursive minimizations building on the solution presented by \cite{Hammersley74} for the case of bifurcating trees. Our proof of the conjecture here, however, is based on an entirely different approach. In particular, our phylogenetic maximization problem, tracing to \cite{Harding1971}, is solved via connections to Huffman trees from information theory~\cite{Glassey76, Huffman52} and to a tape-merging algorithm from the early years of computer science~\cite{Glicksman65}. In Section \ref{app:huffman}, after reviewing the Huffman algorithm and its associated $H$-trees, we discussed the $r$-merge operator and normal sequences (Definition \ref{def:normal_seq}), and we showed that the $H$-tree produced with the uniform weight vector on tree leaves is unique (Proposition \ref{prop:h_tree_shape}). In Section \ref{app:completing}, the final steps in the proof consider weight sequences for an $H$-tree and connect the $H$-tree to majorization theory (Lemma \ref{lemma:hu_majorization}) before finally showing that the unique $H$-tree with a uniform weight vector is the conjectured unique maximally probable $r$-furcating unlabeled topology (Theorem \ref{thm:r_furcating_max_shape}). 

The information-theoretic connection that we have uncovered here for a mathematical phylogenetics problem merits further investigation. Huffman trees are grounded in the construction and analysis of prefix-free codes, in which each symbol from a source alphabet is encoded by a codeword in such a way that no codeword is a prefix of any other codeword~\cite[Section 5.6, p.~92]{Cover05}. In a setting with a uniform weight vector, the optimal prefix-free code for an alphabet with a uniform probability distribution can be extracted from the tree constructed by the Huffman algorithm, where an optimal code is one that minimizes the expected codeword length. Connections between the setting of evolutionary trees and information theory might be of interest in light of recent analyses of various tree encodings in mathematical phylogenetics~\cite{ChauveEtAl25, ColijnAndPlazzotta18, KimEtAl20:pnas, MarancaAndRosenberg24, PennEtAl25}.

In Section $\ref{app:withsim}$, we extended our analysis to $r$-furcating trees with simultaneity, presenting a conjecture for the maximally probable $r$-furcating unlabeled topology in the setting in which labeled histories permit simultaneity (Conjecture \ref{conj:max_r_sim}), with Table \ref{tab:tri_opt} confirming the conjecture for values for $3 \leq n \leq 29$ for the trifucating case. This conjecture provides a natural extension of a previous conjecture in the bifurcating case (Conjecture 14 of \cite{Dickey26}), claiming that the maximally probable $r$-furcating unlabeled topology in the non-simultaneous case is also the maximally probable $r$-furcating unlabeled topology with simultaneity.

\bigskip
\noindent {\bf Acknowledgments.} Grant support was provided by National Science Foundation grant DMS-2450005.

\small
\bibliographystyle{plain}
\bibliography{multifurcating2}
\clearpage


\end{document}